\documentclass[12pt,review]{elsarticle}

\usepackage[margin=2.2cm]{geometry}
\usepackage{mathrsfs}
\usepackage{amssymb}
\usepackage{amsbsy}
\usepackage{amsmath}
\usepackage[mathscr]{euscript}
\usepackage{amssymb}
\newtheorem{proposition}{Proposition}

\newproof{proof}{Proof}
\newdefinition{definition}{Definition}
\newdefinition{example}{Example}
\journal{Journal of Geometry and Physics}

\begin{document}

\begin{frontmatter}

%% Title, authors and addresses

%% use the tnoteref command within \title for footnotes;
%% use the tnotetext command for theassociated footnote;
%% use the fnref command within \author or \address for footnotes;
%% use the fntext command for theassociated footnote;
%% use the corref command within \author for corresponding author footnotes;
%% use the cortext command for theassociated footnote;
%% use the ead command for the email address,
%% and the form \ead[url] for the home page:
%% \title{Title\tnoteref{label1}}
%% \tnotetext[label1]{}
%% \author{Name\corref{cor1}\fnref{label2}}
%% \ead{email address}
%% \ead[url]{home page}
%% \fntext[label2]{}
%% \cortext[cor1]{}
%% \address{Address\fnref{label3}}
%% \fntext[label3]{}

\title{Deformed Hamiltonian vector fields and Lagrangian fibrations}

%% use optional labels to link authors explicitly to addresses:
%% \author[label1,label2]{}
%% \address[label1]{}
%% \address[label2]{}

\author{David S. Tourigny}
\ead{dst27@cam.ac.uk}
\address{Department of Applied Mathematics \& Theoretical Physics \\ University of Cambridge \\ Wilberforce Road \\ Cambridge CB3 0WA, United Kingdom}

\begin{abstract}
%% Text of abstract
Certain dissipative physical systems closely resemble Hamiltonian systems in $\mathbb{R}^{2n}$, but with the canonical equation for one of the variables in each conjugate pair rescaled by a real parameter. To generalise these dynamical systems to symplectic manifolds in this paper we introduce and study the properties of deformed Hamiltonian vector fields on Lagrangian fibrations. We describe why these objects have some interesting applications to symplectic geometry and discuss how their physical interpretation motivates new problems in mathematics.
\end{abstract}

\begin{keyword}
%% keywords here, in the form: keyword \sep keyword
Symplectic geometry \sep Dynamical systems \sep Dissipative systems
%% PACS codes here, in the form: \PACS code \sep code

%% MSC codes here, in the form: \MSC code \sep code
%% or \MSC[2008] code \sep code (2000 is the default)

\end{keyword}

\end{frontmatter}

%% \linenumbers

%% main text
\section{Introduction}

Symplectic geometry arises as the natural generalisation of conservative Hamiltonian mechanics to differentiable manifolds. The phase space of a Hamiltonian system is generalised to a symplectic manifold and phase portraits are interpreted as integral curves of a Hamiltonian vector field. Symplectic geometry therefore has its origins in classical physics, but more recent times have seen string theory play a role in the discovery of Gromov-Witten invariants and the birth of Floer theory. Together with mirror symmetry, these developments are some of the great success stories of symplectic geometry that can be partially attributed to mathematical physics. Certain dissipative dynamical systems arising in physics are also described using a symplectic viewpoint although not in the setting of differentiable manifolds \cite{Duf62,Tou15}. The current paper grew out of an attempt to put these dynamical systems into the context of symplectic geometry.  

Hamiltonian vector fields, which generalise dynamical systems appearing in classical mechanics, play a central role in several different versions of Floer theory for symplectic manifolds and Lagrangian submanifolds. In particular, the original motivation for Floer's work was to find a proof for Arnold's conjecture that the number of periodic solutions of a Hamiltonian system on a symplectic manifold is bounded below by the sum of its Betti numbers. Hamiltonian vector fields also generate a group of exact symplectomorphisms that determine the geometry of a symplectic manifold. From a different viewpoint, these mathematical abstractions provide a geometric interpretation for many physical arguments, such as preservation of the phase space distribution function in Liouville's theorem or conservation of energy along the integral curves of a Hamiltonian vector field. In light of this it is quite remarkable that Hamiltonian vector fields have such a clear physical interpretation whilst at the same time motivating (and being used as tools to solve) so many mathematical problems arising in symplectic geometry. Then again, perhaps this is not so surprising given that symplectic geometry was developed to accommodate Hamiltonian systems into a geometric setting. Can something similar be achieved for the dynamical systems considered in \cite{Duf62,Tou15}? This is the problem that we attempt to address here.          

In $\mathbb{R}^{2n}$ our dynamical systems closely resemble Hamilton's, but with the equation for one of the variables in each conjugate pair of coordinates rescaled by a nonzero factor of $q \in \mathbb{R}$  
\begin{equation} \label{start}
\dot{x}_i = q^{-1} \frac{\partial H }{\partial y_i } , \quad \dot{y}_i = - \frac{\partial H }{\partial x_i } .
\end{equation} 
Generally the smooth function $H: \mathbb{R}^{2n} \to \mathbb{R}$ might also depend on $q$. Equations of this form were considered a long time ago in physics and used to model dissipative phenomena \cite{Duf62}. There is also a more recent interpretation of these systems in theoretical biology \cite{Tou15}. It is clear that (\ref{start}) becomes an ordinary Hamiltonian system in the limit $q \to 1$ and so to generalise such dynamical systems to symplectic manifolds we introduce the notion of a deformed Hamiltonian vector field. After proving some basic properties of deformed Hamiltonian vector fields we discuss how these objects are related to certain topics in symplectic geometry.  

\section{Preliminaries}

The purpose of this section is two-fold. Firstly, to provide some physical motivation for deformed Hamiltonian vector fields. Secondly, to introduce the geometric objects that we will be studying throughout this paper, before fixing notation and conventions for the proceeding sections. The exposition in subsection 2.2 will be at a level suitable for those familiar with basic differential geometry and algebraic topology, and requires no previous exposure to symplectic geometry, mirror symmetry or Floer theory.

\subsection{Deformed classical mechanics}    
In \cite{Duf62}, Duffin studied systems of the form (\ref{start}) as a proposed model for certain dissipative physical phenomena. He called the dynamics {\em psuedo-Hamiltonian}. The same equations have also been shown to include models of very simple biological processes \cite{Tou15}. Properties of these systems are clearly different from those of conservative Hamiltonian systems since, reflecting dissipation, the Hamiltonian $H$ is not an integral of motion
\begin{equation} \label{divH}
\frac{dH}{dt} = (q^{-1} - 1)\sum_{i=1}^n \frac{\partial H}{\partial x_i}\frac{\partial H}{\partial y_i} . 
\end{equation}          
Duffin introduced a deformed Poisson bracket
\begin{equation}
\sum_{i=1}^n q^{-1} \frac{\partial H}{\partial x_i}\frac{\partial F}{\partial y_i} - \frac{\partial F}{\partial x_i}\frac{\partial H}{\partial y_i}
\end{equation}
to mirror the formalism of conventional classical mechanics. He showed that this bracket remains invariant under specific representations of canonical transformations, and in \cite{Tou15} these were discussed in the setting of q-deformed groups. Duffin also proved a dissipative version of Liouville's theorem for statistical ensembles on phase space under the assumption that second derivatives $\partial^2 H/\partial x_i \partial y_i$ are constant. One purpose for studying deformed Hamiltonian vector fields will be to generalise some of these results to a coordinate-free framework.

As $q$ varies it traces out a one-parameter family of equations (\ref{start}). We can ask how the typical dynamics vary as a function of $q$, and as a first step in doing so consider (\ref{divH}) under the constraint that $H$ is locally positive-definite with $\sum_{i=1}^n \frac{\partial H}{\partial x_i}\frac{\partial H}{\partial y_i} >0$ for all $t>0$. We then have three possible regimes for the parameter $q$. If $q<0$ or $q>1$ then the coefficient $(q^{-1}-1)$ is negative and $dH/dt$ decreases with $t$ so that $H$ is a true Lyapunov function for the dynamics (\ref{start}). Conversely, if $0<q<1$, the coefficient $(q^{-1}-1)$ is positive and it is $-H$ that plays the role of a Lyapunov function. At $q = 1$ the dynamics are Hamiltonian by construction, but $dH/dt$ becomes singular at the point $q=0$. We can therefore expect that $q$ acts as a bifurcation parameter for equilibria in these particular systems, and almost surely more complicated behaviour may exist when we relax the assumptions on $H$. This leads to the consideration of three important limits. The Hamiltonian limit, $q \to 1$, has obvious implications for the system (\ref{start}). The limits $q \to 1$ and $q \to 0$ are more complicated to analyse, but from the above argument most likely represent bifurcations of the dynamics. By introducing deformed Hamiltonian vector fields we shall associate interesting geometric interpretations to each of these limits.

\subsection{Geometric background}
Throughout $(M, \omega)$ will denote a differentiable manifold $M$ of dimension $2n$ equipped with a closed and non-degenerate $2$-form $\omega$. Symplectic manifolds always admit an almost complex structure, i.e. an automorphism $J:TM \to TM$ of the tangent bundle satisfying $J^2=-id$, and $J$ is said to be {\em compatible} with $\omega$ if $G(\cdot, \cdot)= \omega(\cdot, J \cdot)$ is a Riemannian metric on $M$. We call $G$ the {\em standard Riemannian metric associated with $J$}. If the Nijenhuis tensor associated with $J$ vanishes then $J$ is said to be {\em integrable} and $(M, J)$ {\em complex}. If $J$ is both integrable and compatible with $\omega$ then the triple $(M,\omega, J)$ is called {\em K\"{a}hler} and the induced metric $G$ is called a {\em K\"{a}hler metric}. One can think of the triple $(\omega, J, G)$ on an equal footing, taking a K\"{a}hler metric as the starting point and varying the complex or symplectic structures independently.         

Recall that a submanifold $L \subset M$ of a symplectic manifold $(M,\omega)$ is called {\em Lagrangian} if $L$ is half the dimension of $M$ and $\omega$ vanishes when restricted to $L$. A theorem of Weinstein says that a sufficiently small neighbourhood of a Lagrangian submanifold $L \subset M$ can always be identified with a neighbourhood of the zero section in $T^*L$ by a diffeomorphism that preserves the symplectic form (i.e., a symplectomorphism). By a Lagrangian fibration $\pi: (M,\omega) \to B$ we mean a smooth fibration $\pi : M \to B$ over an $n$-dimensional base manifold $B$ such that at every point $x \in B$ the fibre $F_x = \pi^{-1}(x)$ is a Lagrangian submanifold of the symplectic manifold $(M, \omega)$. The obvious noncompact examples are cotangent bundles $\pi : T^*B \to B$ where the zero section is canonically identified with $B$, but it is rare to find particularly exotic examples of compact Lagrangian fibrations without singular fibres. The Arnold-Liouville theorem says that locally a Lagrangian fibration with compact, connected fibre is affinely isomorphic to the product of an affine space with a torus. Indeed, each compact, connected fibre of a smooth Lagrangian fibration must necessarily be a torus and the base must have canonical {\em integral affine structure}. This means that $B$ admits an atlas of coordinate charts whose transition functions are elements of the affine group $\mathbb{R}^n\rtimes GL(n, \Bbb Z)$. 

After choosing a compatible almost complex structure $J$ on the total space of a Lagrangian fibration $\pi: (M,\omega) \to B$ the standard Riemannian metric $G$ induces a decomposition of the tangent bundle $TM$ into vertical and horizontal subspaces
\begin{equation} \label{tandecomp}
TM = T^BM \oplus  T^FM .
\end{equation}
The subspace $T^FM$ is the tangent space to the fibres of $\pi: (M,\omega) \to B$ and $T^BM$ is its $G$-orthogonal complement. This in turn corresponds to a decomposition of the metric
\begin{equation}
G = G_B \oplus G_F ,
\end{equation}
where $G_B$ can often be identified with the pull-back under the projection of some Riemannian metric on $B$ (that we also call $G_B$ when it is understood). $G_F$ is the part that annihilates the orthogonal complement of the fibres. As above we prefer to speak of the choice of almost complex structure determining $G$, but it will sometimes be convenient to view the almost complex structure as being determined by a choice of metric on $B$. One such example is the analogue of the Sasaki metric $G^{Sas}$ \cite{Sas58} for the cotangent bundle $T^*B$ of a Riemannian manifold $(B,G_B)$, which uniquely determines an almost complex structure $J^{Sas}: G^{Sas}(\cdot , \cdot ) = \omega(\cdot , J^{Sas} \cdot)$. Here $\omega = d \theta$ is the canonical symplectic form where $\theta$ is the tautological $1$-form on the cotangent bundle $T^*B$. The pair $(T^*B, d \theta)$ is naturally a symplectic manifold and the fibres of $\pi: (T^*B,d \theta) \to B$ are Lagrangian submanifolds.  

Alongside the decomposition of $TM$ induced by the choice of $J$ there is a corresponding decomposition of the cotangent bundle 
\begin{equation}
T^*M = ( T^BM)^* \oplus ( T^FM)^* ,
\end{equation}
where $( T^BM)^*$ is the annihilator of $T^FM$ and $( T^FM)^*$ is that of $T^BM$. This induces a bigrading on differential forms of degree $a$
\begin{equation}
\Omega^a(M) = \bigoplus_{b+c=a} \Omega^{b,c}(M) ,
\end{equation}
with $\Omega^{b,c}(M)$ denoting the space of sections of $ \wedge^b  ( T^BM)^* \otimes  \wedge^c ( T^FM)^*$. Whenever there is such a splitting of differential forms the de Rham differential $d$ can be written as a sum of four components 
\begin{equation}
d = d_{1,0} + d_{0,1} + d_{2,-1} + d_{-1,2} ,
\end{equation}
where $d_{c,d} : \Omega^{a,b}(M) \to \Omega^{a+c,b+d}(M)$. We say that $\alpha \in \Omega^a(M)$ is of type $(b,c)$ if $\alpha \in \Omega^{b,c}(M)$. The Lagrangian condition together with non-degeneracy of the symplectic form implies $\omega$ is of type $(1,1)$. For integrability reasons the operator $d_{-1,2} $ vanishes when $\pi : M \to B$ is a smooth fibration so that after dropping the annoying indices by defining
\begin{equation}
\delta := d_{2,-1} , \quad \partial_+: = d_{1,0}  \quad \mbox{and} \quad \partial_- := d_{0,1} 
\end{equation}
the exterior derivative reduces to
\begin{equation} \label{decom}
d = \partial_+ + \partial_- + \delta .
\end{equation}
Using $d^2=0$ one obtains the relations
\begin{equation} \label{rule1}
\partial_-^2 = \delta^2= \partial_+ \partial_- + \partial_- \partial_+ = 
\partial_+ \delta + \delta \partial_+ =  \partial_+^2  + \partial_-  \delta + \delta \partial_- = 0 .
\end{equation}
The identity $\partial^2_-=0$ is attributed to the fact we have an involutive distribution on $M$ induced by the vertical directions of the fibration. Obstruction to the identity $\delta = 0$ comes down to the fact that the ($G$-orthogonal) complementary distribution might not necessarily be integrable. If it were, $M$ would admit a pair of transversal Lagrangian foliations that, although entirely possible, is a rather strict condition to impose. Manifolds with this property have been called {\em bi-Lagrangian}, {\em para-K\"{a}hler} or {\em $\mathbb{D}$-K\"{a}hler} in the literature \cite{Eta06,Cru96,Har12,Cor04}. In this paper however, we shall reserve the phrase {\em bi-Lagrangian} for integrability of the $J$-induced complementary distribution of an existing Lagrangian fibration, i.e. $(M,\omega,J)$ is bi-Lagrangian if and only if $\delta = 0$. 

Given a smooth function $H : M \to \mathbb{R}$ the Hamiltonian vector field $X_H \in TM$ on $(M, \omega)$ is the unique vector field defined by
\begin{equation}
\omega(X_H, \cdot ) = -dH .
\end{equation}
By Liouville's theorem $X_H$ generates an (exact) symplectomorphism of $M$ because its flow preserves the symplectic form
\begin{equation}
L_{X_H} (\omega) = d (\omega(X_H, \cdot )) = - d^2H = 0,
\end{equation}
where $L_{\xi}$ denotes the Lie derivative along the flow of the vector field $\xi$. $X_H$ uniquely defines a gradient vector field because of the fact that
\begin{equation}
G(JX_H, \cdot) = \omega(X_H, \cdot ) = -dH ,
\end{equation}
and so one may identify $JX_H$ with $-\nabla H$, the gradient of $-H$ taken with respect to the standard Riemannian metric associated with $J$. If $H(t)=H(t+1) : M \to \mathbb{R}$ defines a 1-periodic family of functions parameterised by $t \in S^1$ then it generates a family of exact symplectomorphisms $\phi_t : M \to M$ via 
\begin{equation}
\frac{d}{dt} \phi_t = X_{H(t)} \circ \phi_t , \quad \phi_0 = id .
\end{equation}
The {\em Arnold conjecture} states that for $M$ closed the number of non-degenerate 1-periodic solutions of the associated differential equation
\begin{equation} \label{ode}
\dot{z}(t) = X_{H(t)}(z(t)) ,
\end{equation}
is bounded below by the sum of the Betti numbers of $M$. 

\section{Deformed Hamiltonian vector fields}

We are now in a position to define the objects of primary interest to this paper. After introducing deformed Hamiltonian vector fields we will prove several properties that explain how they are related to their ordinary Hamiltonian counterparts.  

As before, let $\pi:(M,\omega) \to B$ be a Lagrangian fibration of a symplectic manifold $(M, \omega)$ and pick a smooth function $H: M \to \mathbb{R}$. Choose an almost complex structure $J$ on $M$ compatible with $\omega$ and consider the natural decomposition of the tangent bundle and standard metric
\begin{equation}
T^*M = T^BM \oplus  T^FM , \quad G = G_B \oplus G_F .
\end{equation}
The one-parameter family of metrics $\{G_q\}$ is formed by rescaling the metric in the fibre direction so that for each fixed value of $q \in (0,1]$ we have a Riemannian metric
\begin{equation}
G_q =  G_B \oplus q G_F  
\end{equation}
(we postpone the discussion of what happens for negative $q$ until the next section). Then $\{(M,G_q)\}$ defines a family of Riemannian manifolds with fibres whose volumes are monotonically decreasing as $q \to 0$. However, as before we prefer to view $\{G_q\}$ as being determined by the almost complex structures $\{J_q\}$ and consider the family $\{(M,\omega,J_q)\}$ defined by requiring that $G_q(\cdot,\cdot) = \omega(\cdot, J_q\cdot)$ for each $q \in (0,1]$. Using the decomposition of the exterior derivative induced by the Lagrangian fibration we also introduce a family of operators $\{d_q\}$ to go alongside this family of degenerating symplectic manifolds. 
\begin{definition}
For fixed $q \in (0,1]$ the deformed exterior derivative $d_q$ is given by
\begin{equation}
d_q :=  \partial_+ + q^{-1} \partial_-   + q \delta .
\end{equation} 
\end{definition}
The following proposition confirms that for each $q \in (0,1]$ the operator $d_q$ is a well-defined differential on $\Omega^*(M)$.
\begin{proposition}
\emph{$d_q^2 = 0$.}
\label{}
\end{proposition}

\begin{proof}
We have $d_q^2  = \partial^2_+ +\partial_-\delta + \delta \partial_- +q(\partial_+ \delta + \delta \partial_+) + q^{-1} (\partial_-\partial_+ + \partial_+\partial_-) + q^2 \delta^2 + q^{-2} \partial_-^2$ and by (\ref{rule1}) every term multiplying a given power of $q$ vanishes. \qed
\end{proof}
It must be emphasised that the definition of $d_q$ is only possible because we have a decomposition of the exterior derivative (\ref{decom}) that depends on the Lagrangian fibration {\em and also the choice of almost complex structure $J$}. Therefore the two families $\{d_q\}$ and $\{J_q\}$ are not independent and when we refer to one element, $d_q$, say, we will always have a corresponding object, $J_q$, in the other family. It is important to bear this in mind since this leads to two equivalent definitions of a deformed Hamiltonian vector field. 
\begin{definition}
The deformed Hamiltonian vector field generated by $H$ is the unique vector field $X^q_H \in TM$ that satisfies
\begin{equation}
\omega(X^q_H, \cdot ) = -d_qH .
\end{equation}
\end{definition}
This generalises the usual definition of a Hamiltonian vector field since $q$ serves as a ``deformation parameter'' for the exterior derivative in the sense that we return to the classical definition in the limit $q \to 1$. Once more we have actually defined an entire family $\{X^q_H\}$ parameterised by $q \in (0,1]$ and by writing $X^q_H$ we are referring to the deformed Hamiltonian vector field corresponding to $d_q$ and $J_q$. The next proposition provides an equivalent definition for $X^q_H$ in terms of the metric $G_q$. 
\begin{proposition}
Given a deformed Hamiltonian vector field $X^q_H$, the vector field $JX^q_{H_q}$ is the gradient of $-H$ defined using the metric $G_q$.
\label{grad}
\end{proposition}

\begin{proof}
We want to show that $G_q(JX^q_H, Y) = -dH(Y)$ for all $Y \in TM$. Using the decomposition of $TM$ we write the vector field $Y \in TM$ as $Y = Y_+ + Y_-$, where $Y_+ \in T^BM$ and $Y_- \in T^FM$, and define the new vector field $Y^q$ by setting $Y^q_+ = q^{-1}Y_+$ and $Y^q_- = Y_-$. Note $\delta =0$ when acting on functions so that 
\begin{equation}
-d_qH(Y^q) = -q^{-1} \partial_+H(Y_+) - q^{-1}\partial_-H(Y_-) = -q^{-1}dH(Y)
\end{equation}
and using $\omega$-compatibility of $J$ we have
\begin{equation}
\omega(X^q_H,Y^q) = q^{-1}G_B(JX^q_H,Y_+) + G_F(JX^q_H,Y_-) .
\end{equation}
Using the definition of a deformed Hamiltonian vector field we can equate both expressions above and multiply by $q$ to obtain
\begin{equation}
-dH(Y) = G_B(JX^q_H,Y_+) + qG_F(JX^q_H,Y_-) = G_q(JX^q_H,Y) .
\end{equation}
\qed
\end{proof}
Thus, the vector field $JX^q_H$ on the manifold $(M,\omega, J)$ is defined to be the vector field that would be a gradient with respect to the standard Riemannian metric on the manifold $(M,\omega,J_q)$. That is to say, $JX^q_H= - \nabla_qH$ where $\nabla_q$ is the gradient associated with $G_q$. Although somewhat more convoluted, this definition makes explicit the choice of almost complex structure in the construction of a deformed Hamiltonian vector field. Also note that the bijection $Y \to Y^q$ used in the proof of Proposition \ref{grad} maps any Hamiltonian vector field $X_H$ to the corresponding deformed Hamiltonian vector field $X^q_H$, providing yet another equivalent definition.    

The first definition of a deformed Hamiltonian vector field is more natural from the perspective of understanding the flow of $X^q_H$ and also because geometric properties of the Lagrangian fibration $\pi: (M,\omega) \to B$ are reflected in the analytic properties of $d_q$. It turns out that these properties are tied up with the particular choice of function $H$ used to generate the deformed Hamiltonian vector field. We will now describe what this means.
\begin{definition}
Functions $H: M \to \mathbb{R}$ satisfying the property $\partial_- \partial_+ H = 0$ are called simple, whilst functions satisfying $\partial_+ H = 0$ are called exceptionally simple. 
\end{definition}
It is obvious that exceptionally simple implies simple, but the converse is not true. The exceptionally simple condition is intrinsic to the fibration whereas the simple condition depends on the choice of almost complex structure. Sometimes it will prove useful to decompose the function $H$ as $H=H_+ + \hat{H} + H_-$ where $\partial_{\pm}H_{\mp} = 0$. $\hat{H}$ is the part of $H$ that is not necessarily simple nor exceptionally simple, and in particular one has that $\partial_+ \partial_- H = \partial_+ \partial_- \hat{H}$ since $H_+ + H_-$ is simple. Of course this decomposition is not unique, but we assume it is ``maximal'' in the sense that $\hat{H} = 0 $ whenever possible. To get a feel for what the simple condition really means we choose a Darboux coordinate chart $\{x_i,y_j \}$ for $T^*\mathbb{R}^n$ as a model for the Lagrangian fibration $(M,\omega,J)$ in which $\{x_i\}$ are coordinates on the base $\mathbb{R}^n$ and $\{y_i\}$ are coordinates on the fibres. A generic Hamiltonian is just an arbitrary function $H(x,y)$ of all the coordinates and one finds that
\begin{equation} \label{local}
\partial_- \partial_+ H(x,y) = \sum_{i,j=1}^n \frac{\partial^2 H}{\partial y_i \partial x_j} dy_i \wedge dx_j ,
\end{equation}
so that $H$ being simple is equivalent to $H(x,y) = H'(x) + H''(y)$. Likewise, $H$ being exceptionally simple is equivalent to setting $H(x,y)=H''(y)$ as a function of the fibre coordinates only. The following proposition describes how the flow of $X^q_H$ depends on the choice of Hamiltonian $H$ by answering the question of when a deformed Hamiltonian field generates a symplectomorphism.
\begin{proposition}
For $q \neq 1$ a deformed Hamiltonian vector field $X^q_H$ on $(M,\omega,J)$ is symplectic if $H$ is of the form $H= H_+ + H_-$ with $\partial_{\pm}H_{\mp} = 0$. If, in addition, $(M,\omega,J)$ is bi-Lagrangian then the flow of $X^q_H$ is symplectic if and only if $H$ is simple. 
\label{simpletest}
\end{proposition}    

\begin{proof}
After a straightforward calculation it becomes clear that in general $X^q_H$ does not generate a symplectomorphism unless $q=1$ since the $1$-form $d_qH$ is not necessarily closed
\begin{equation}
L_{X^q_H}(\omega) = -dd_qH = (q^{-1}-1)(\partial^2_+ + \partial_- \partial_+) H  .
\end{equation} 
The $2$-forms $\partial^2_+H$ and  $\partial_- \partial_+ H$ are of different type and so we can not have $-\partial^2_+H = \partial_- \partial_+  H$ unless both are zero, hence proving that $H$ must be simple if $\partial^2_+=0$. Using relations (\ref{rule1}), if $H = H_+ + H_-$ then $\partial^2_+H=\partial^2_+ H_+ = -\delta \partial_-H_+ = 0$ and so this condition is sufficient whenever $(M,\omega,J)$ is not bi-Lagrangian.   
\qed
\end{proof}

We may also ask when a deformed Hamiltonian vector field is {\em conformally symplectic}, i.e. generates a conformally symplectic diffeomorphism $\phi: M \to M$ that preserves the symplectic form up to some constant $1 \neq c \in \mathbb{R}$. This would serve as a generalisation of Duffin's dissipative version of Liouville's theorem \cite{Duf62}. The conformal symplectomorphisms form a group that, like the group of symplectomorphims, is one of Cartan's six classes of groups of diffeomorphisms  on a manifold $M$.  Conformally symplectic vector fields have also previously been used to generalise simple mechanical systems with dissipation \cite{Woi98,McL01}. The proposition below answers the question of when a deformed Hamiltonian vector field is conformally symplectic on a bi-Lagrangian manifold.
\begin{proposition}
For $q \neq 1$ a deformed Hamiltonian vector field $X^q_H$ on bi-Lagrangian $(M, \omega, J)$ is conformally symplectic if and only if $\omega = c' \partial_- \partial_+H$ for some nonzero constant $c' \in \mathbb{R}$.
\label{conf}
\end{proposition}    

\begin{proof}
When $(M,\omega,J)$ is bi-Lagrangian the condition that $X^q_H$ generates a conformal symplectomorphism is that
\begin{equation}
L_{X^q_H}(\omega) = -dd_qH = (q^{-1}-1) \partial_- \partial_+ H  = c \omega \ 
\end{equation} 
for some nonzero constant $c \in \mathbb{R}$. Clearly this implies $\omega = c^{-1}(q^{-1}-1)\partial_- \partial_+ H$. 
\qed
\end{proof}
Thus, on a bi-Lagrangian manifold $(M,\omega,J)$ a deformed Hamiltonian vector field $X^q_H$ is conformally symplectic whenever $\omega = \partial_-\partial_+K$ is defined {\em globally} by the analogue of a K\"{a}hler potential $K: M \to \mathbb{R}$ with $H = H_+ + K + H_-$ satisfying $ \partial_{\pm}H_{\mp} = 0$. In Duffin's local treatment this condition manifests itself as the assumptions on $\partial^2 H/\partial x_i \partial y_i$ (compare with Equation \ref{local}). Globally this is yet again a very strict condition to impose on a symplectic manifold since, as in the K\"{a}hler case, when $(M,\omega,J)$ is bi-Lagrangian $\omega$ is usually only determined by a potential {\em locally} \cite{Cor04}. Examples of these manifolds do exist however. Note that because $\omega$ is necessarily of type $(1,1)$ Proposition \ref{conf} breaks down when $(M,\omega,J)$ is not bi-Lagrangian unless we impose the additional condition that $\partial^2_+H$ vanishes. We can not ask for $H$ to be exceptionally simple (our definition of a conformally symplectic vector field excludes the symplectic case), so $H$ must be a non-simple Hamiltonian that satisfies $\partial^2_+H = 0$ with $\omega = c' \partial_- \partial_+H$. This further restricts the types of functions that may be considered. 

Deformed Hamiltonian vector fields can be used to induce an algebraic structure over the differentiable functions on $M$. It is well-known that the Poisson bracket induces the structure of a Lie algebra on $\Omega^0(M)$. In 1948 A. A. Albert introduced the concept of a {\em Lie-admissible algebra} \cite{Alb48}, defined for an algebra $\mathcal{U}$ whose commutator algebra (the anti-commutator algebra with multiplication $u*v = uv -vu$ for all $u,v \in \mathcal{U}$) admits the structure of a Lie algebra. Following Duffin's work on the deformed Poisson bracket, this motivates the following definition.

\begin{definition}
The bracket $\{ \cdot , \cdot \}_q : \Omega^0(M) \times \Omega^0(M) \to \Omega^0(M)$ is defined by $\{H,F\}_q = \omega(X^q_H,X_F)$. 
\end{definition}
We can then prove.
\begin{proposition}
The algebra of smooth functions on $M$ equipped with multiplication $H *F \equiv \{ H,F\}_q$ for $H,F \in \Omega^0(M)$ is Lie-admissible. 
\label{alg}
\end{proposition}

\begin{proof}
We need to show that $\omega(X^q_H,X_F) - \omega(X^q_F,X_H)$ satisfies the Jacobi identity. Explicitly, we have
\begin{equation}
\omega(X^q_H,X_F) - \omega(X^q_F,X_H) = -d_qH(X_F) -dH(X^q_F)  .
\end{equation} 
Using the identification of $X^q_F$ with $X_F$ as in the proof of Proposition \ref{grad} the right-hand side can be decomposed and individual terms regrouped to yield 
\begin{equation}
\omega(X^q_H,X_F) - \omega(X^q_F,X_H) = (1 + q^{-1}) \omega(X_H,X_F)  .
\end{equation}
This bilinear form is therefore identified with $(1+q^{-1})$ times the canonical Poisson bracket, which satisfies the Jacobi identity following classical results in symplectic geometry. 
\qed
\end{proof}
Unlike the Lie subalgebra formed by Hamiltonian vector fields however, there is no subalgebra that may be easily constructed from deformed Hamiltonian vector fields. This is because any such algebra does not close under the action of the Lie bracket.  
  
As in the Hamiltonian case, a deformed Hamiltonian vector field $X^q_{H} \in TM$ determines a differential equation
\begin{equation} \label{ode2}
\dot{z}(t) = X^q_H(z(t)) , 
\end{equation}
which is the appropriate generalisation of (\ref{start}). With a time-dependent Hamiltonian $H: S^1 \times M \to \mathbb{R}$ there is an associated two-parameter family of diffeomorphisms $\phi^q_t : M \to M$ generated via
\begin{equation}
\frac{d}{dt} \phi^q_t = X^q_{H(t)} \circ \phi^q_t , \ \ \ \phi^q_0 = id ,
\end{equation}
for each value of $q \in (0,1]$. These are symplectomorphisms when $H(t)=H_+(t)+H_-(t)$ (or conformal symplectomorphisms when $H(t)$ and $(M,\omega,J)$ satisfy the requirements of Proposition \ref{conf}), but in general they do not preserve $\omega$ unless $q=1$. It would therefore not be prudent to formulate a deformed analogue of the Arnold conjecture for solutions to a time-dependent version of (\ref{ode2}). However, when $H$ is independent of time the Arnold conjecture follows trivially from the fact that the critical points of $H$ are constant solutions of (\ref{ode2}) and therefore 1-periodic. Moreover, the fact that deformed Hamiltonian vector fields may reduce to ordinary Hamiltonian vector fields on certain domains of $H$ suggests an analogue of Floer theory might apply to particular submanifolds of $(M,\omega)$. This will be expanded upon in the next section. It therefore seems plausible to see how far one can get following the approach of Floer and studying solutions of the partial differential equation
\begin{equation} \label{qFloer}
\frac{\partial u}{\partial s} + J_t(u) \left( \frac{ \partial u}{\partial t} - X^q_{H(t)}(u) \right) = 0 ,
\end{equation}
for smooth maps $u : \Sigma \to M$ from a Riemann surface $\Sigma$ with appropriate boundary conditions. If $H$ does not depend on time then time-independent solutions to the deformed Floer equation (\ref{qFloer}) satisfy
\begin{equation}
\frac{du}{ds} + \nabla_q H(u) =0 .
\end{equation}
These trajectories are flows of the gradient of $-H$ defined with respect to the deformed metric $G_q$. In the next section we hint at a model for deformed Floer theory that is designed to expand upon this point. Namely, we consider a finite-dimensional gradient flow problem on a Lagrangian fibration equipped with the metric $G_q$. 

\section{Applications to symplectic geometry}

\subsection{Para-complex geometry and mirror symmetry}

Important examples of Lagrangian fibrations are Lagragian torus fibrations, and for the key ideas behind the classification of these the reader is referred to \cite{Dui80,Zun03}. From the Arnold-Liouville theorem a smooth Lagrangian fibration $\pi : (M,\omega) \to B$ with connected, compact fibres is necessarily a torus fibration over an integral affine manifold $B$ with transition functions in the subgroup $\Bbb R^n \rtimes GL(n,\Bbb Z) \subset \mbox{Aff}(\Bbb R^n)$. The integral affine structure determines a subbundle $\Lambda^* \subset T^*B$ of integral 1-forms and the holonomy of $\Lambda^*$ is called the affine monodromy of the Lagrangian torus fibration. The fibration $\pi : M \to B$ is a principal torus bundle if and only if the affine monodromy is trivial and globally there exists an isomorphism $M \cong T^*B/\Lambda^*$ if and only if $\pi : M \to B$ admits a global section. In symplectic coordinates on a Lagrangian torus fibration the metric $G_q$ is therefore
\begin{equation}
G_q = (G_B)_{ij} dx_i \otimes dx_j + q(G_B^{-1})_{ij} dy_i \otimes dy_j
\end{equation} 
and we find that the diameter of $M$ stays bounded whilst the volume of the fibres shrink to zero as $q \to 0$. Translating this to the family of almost complex structures $\{J_q\}$ we recognise the limit $q \to 0$ as the large complex structure limit of mirror symmetry (see \cite{Str96,Leu05,Kon01} and references therein). Thus, as suggested in subsection 2.1, one may assign a geometric interpretation to this limit. 

For the definition of deformed Hamiltonian vector fields it seemed more natural to assume that $q>0$, but often one expects to have $q<0$. In this case the metric $G_q$ is no longer Riemannian but instead a psuedo-Riemannian metric with neutral signature. In fact when $q= -1$ the metric 
\begin{equation}
G_{-1}= G_B \oplus -G_F 
\end{equation}  
is precisely the standard metric induced by a choice of {\em almost $\mathbb{D}$-complex structure} $T$ on $M$ (here we use the terminology of Harvey and Lawson \cite{Har12} whilst others call $T$ an {\em almost bi-Lagrangian}, {\em para-complex}, or an {\em almost product structure}). In analogy with the complex case an almost $\mathbb{D}$-complex structure $T$ is an automorphism $T:TM \to TM$ satisfying $T^2=id$ with $G_{-1}(\cdot , \cdot ) = \omega(\cdot , T \cdot) $ the {\em standard psuedo-Riemannian metric associated with $T$}. Here the specific choice of $T$ is determined by the choice of $J$. In particular, the decomposition (\ref{tandecomp}) of $TM$ induced by $J$ coincides with the eigenspace decomposition of $TM$ induced by $T$. For now we assume that both $J$ and $T$ are integrable. This means that $(M,\omega,J)$ is K\"{a}hler and $(M,\omega,T)$ is $\mathbb{D}$-K\"{a}hler (equivalently $(M,\omega,J)$ is bi-Lagrangian). By allowing negative values of $q$ the family $\{J_q\}$ extended to the interval $q \in [-1,1]$ traces out a path in the combined space of all $\omega$-compatible ($\mathbb{D}$-)complex structures, the $\mathbb{D}$-complex structures compatible in the sense that $\omega(\cdot , T \cdot)$ is a metric of neutral signature on $M$. This path starts at the $\mathbb{D}$-complex structure $T$ with $q = -1$ and ends at the complex structure $J$ with $q=1$. However, it must also pass through the singular point at $q=0$ where the metric $G_q$ degenerates on the fibres of $\pi: (M,\omega) \to B$. As described previously, this point represents a boundary or cusp in the space of compatible complex structures and the limit $q \to 0^+$ is precisely the large complex structure limit of mirror symmetry in which the SYZ conjecture is expected to hold \cite{Kon01}. 

Allowing $q$ to vary across the interval $[-1,1]$ automatically extends semi-flat mirror symmetry to include a duality with $\mathbb{D}$-K\"{a}hler geometry and it turns out that analogues of special Lagrangian submanifolds (the basis of the SYZ conjecture) have already been studied there \cite{Har12}. In particular, it is the Ricci-flat, affine $\mathbb{D}$-K\"{a}hler manifolds that provide the natural duals of Calabi-Yau manifolds and because of their bi-Lagrangian structure these are also Lagrangian torus fibrations over an affine base equipped with Koszul metric. If suitably defined, the parametrisation $\{J_q\}$ should provide a way to move between K\"{a}hler and $\mathbb{D}$-K\"{a}hler Lagrangian fibrations, perhaps as submanifolds in a higher-dimensional ambient space. Mirror symmetry could then be used to set up a quadrality involving mirror pairs of both types of geometry. To the best of our knowledge nothing along these lines has appeared in the literature so far.

\subsection{Morse theory}

It is unlikely that a generic deformed Hamiltonian vector field will admit periodic orbits due to the dissipative nature of its flow. However, Proposition \ref{simpletest} suggests it might still be possible for a non-Hamiltonian $X^q_H$ to display periodic behaviour should there exist a submanifold of $M$ where $\hat{H} = 0$ (recall the decomposition $H=H_+ + \hat{H} + H_-$). Here we sketch out a Morse-type model for the associated Floer theory on cotangent bundles $(T^*B, d\theta)$. This is essentially a conjectural extension of the Lagrange multiplier Morse theory developed in \cite{Fra06,Sch14}. 

Frauenfelder  \cite{Fra06} (and Schecter-Xu \cite{Sch14} for the rank one case) considered Morse theory on the trivial vector bundle $B \times \mathcal{V}^* \to B$ using a smooth function $F : B \times \mathcal{V}^* \to \mathbb{R}$ given by
\begin{equation}
F(x,v^*) = f(x) + v^*(w(x)) ,
\end{equation} 
where $v^* \in \mathcal{V}^*$, $f : B \to \mathbb{R}$ and $w : B \to \mathcal{V}$. Here $\mathcal{V}^*$ is the dual of a finite dimensional vector space $\mathcal{V}$. If $0$ is a regular value of $w$, then it is a well-known fact that there exists a bijective correspondence $\lambda: \mbox{Crit}(F) \to \mbox{Crit}(f|_{w^{-1}(0)})$ between critical points of $F$ and critical points of $f|_{w^{-1}(0)}$. Using several different approaches, both \cite{Fra06} and \cite{Sch14} prove the existence of a homotopy between the moduli spaces of gradient flow lines of $F$ on $B \times \mathcal{V}^*$ and those of $f|_{w^{-1}(0)}$ on $w^{-1}(0)$. Most relevant to us is the adiabatic limit method used in  \cite{Sch14} to show that gradient flow lines of $F$ converge to those of $f|_{w^{-1}(0)}$ as the volume of the fibre is taken to zero. In general, $B \times \mathcal{V}^*$ is of rank $k < n$ so that $w^{-1}(0) \subset B$ is a submanifold of dimension $n -k > 0$. For a cotangent bundle the fibres are always of dimension $n$ however, which means that $w$ must degenerate on certain fibre directions if we are to ensure $n-k$ is nonzero. Even if $f$ is Morse this necessarily implies $F$ can only ever be Morse-Bott so that something must be done to  account for the ``left over'' directions of the fibration.  

The above issue is most easily addressed by perturbing $F$ using a family of Morse functions having compact support on the degenerate directions associated with critical submanifolds. Although $F$ is Morse-Bott its perturbation becomes Morse \cite{Ban13}. It is this approach that realises the decomposition $H = H_+ + \hat{H} + H_-$, with $H_-$ the perturbing Morse function, and $\hat{H}$ and $H_+$ identified with the appropriate generalisations of $v^*(w(x))$ and $f(x)$, respectively. Our assumption on the function $\hat{H}$ is that $\hat{H}= \theta(\hat{w})$ for some $\hat{w} \in T(T^*B)$ whose horizontal projection is a vector field $w: B \to TB$ that has zero set $w^{-1}(0) \subset B$ with codimension $ k$ as a closed, oriented submanifold of $B$. We use $w_i$ to denote the $n$ functions $w_i : B \to \mathbb{R}$ defined by $w$ in an appropriate trivialisation and impose that the vertical projection of $dw$ has rank $k$. The Hamiltonian family $H_q$ is then constructed using a Morse function $H_+ = f: B \to \mathbb{R}$ together with a function $H_- = g$ whose domain will include the critical submanifolds. We assume further that the restriction $f|_{w^{-1}(0)}$ is a Morse function on $w^{-1}(0)$ and extending $g$ to the whole of $T^*B$ using cut-off functions we obtain the Hamiltonian 
\begin{equation}
H_{q}(z) = f(\pi(z))+ \langle z, w(\pi(z)) \rangle + q g(z) , \ \ \ q \in (0,1] ,
\end{equation}
where $z \in T^*B$. The critical point set of $H_0 \equiv H_{q = 0}$ (which is the analogue of $F$ in \cite{Fra06,Sch14}) consists of pairs $(x,y)$ satisfying (in local coordinates)
\begin{equation}
w_i(x) = 0 , \ \ \ df(x)+ y_idw_i(x) = 0 ,
\end{equation}
which by the assumptions on $w$ is just the condition that $x \in w^{-1}(0)$ is a critical point of $f|_{w^{-1}(0)}$. The combination of the $y_i$ spanning the vertical kernel of $dw$ define a $(n-k)$-dimensional fibre $Z_x$ over $x$ that we assume can be extended to a proper fibre bundle $ Z \to w^{-1}(0)$. Because $f|_{w^{-1}(0)}$ is a Morse function with isolated critical points the critical point set of $H_0$ is a disjoint union of isolated critical submanifolds $\mathcal{V}_x \cong Z_x$ that are identified with the fibres of $Z$ over each critical point $x \in \mbox{Crit}(f|_{w^{-1}(0)})$,
\begin{equation}
\mbox{Crit}(H_0) = \coprod_{x \in \mbox{Crit}(f|_{w^{-1}(0)})} \mathcal{V}_x .
\end{equation}
Choosing $g$ to define a family of Morse functions $g_x: Z_x \to \mathbb{R}$ parameterised by $x \in w^{-1}(0)$ means that $H_q$ is a Morse function on $T^*B$ for each $q \in (0,1]$. Critical points $p$ of $H_q$ can be identified with pairs $(x,y)$ where $x \in w^{-1}(0)$ is a critical point of $f|_{w^{-1}(0)}$ and $y$ is a critical point of $g_x$ on the fibre $Z_x$. The index of a critical point $p=(x,y) \in \mbox{Crit}(H_q)$ is
\begin{equation} \label{idx}
\mbox{index}_{H_q}(p) = \mbox{index}_{f|_{w^{-1}(0)}}(x) + \mbox{index}_{g_x}(y) + k .
\end{equation}

By Proposition \ref{grad}, for fixed $q$ we have that $JX^q_{H_q}$ is the negative gradient of $H_q$ defined with respect to the metric $G_q$. Using the flow of 
\begin{equation} \label{flow}
\frac{du}{dt} = JX^q_{H_q}(u) ,
\end{equation}
for each $p \in \mbox{Crit}(H_q)$ we can define the stable and unstable manifolds $W^s_q(p)$ and $W^u_q(p)$, respectively. For $q \in (0,1]$ we assume the pair $(H_q, G_q)$ satisfy the Morse-Smale condition so that for all $p^{\pm} \in \mbox{Crit}(H_q)$ the family of moduli spaces $\mathcal{M}_q(p^-,p^+) = W^u_q(p^-) \cap W^s_q(p^+) / \mathbb{R}$ is a family of smooth manifolds all of dimension $\mbox{dim}(\mathcal{M}_q(p^-,p^+)) = \mbox{index}_{H_q}(p^-) - \mbox{index}_{H_q}(p^+) -1$. Thus, we can define a family of Morse-Smale-Witten complexes, $\mathcal{C}_*(H_q,J_q)$, by counting flow lines of (\ref{flow}) that join critical points of $H_q$.  The notation $\mathcal{C}_*(H_q,J_q)$ indicates the choice of Hamiltonian and almost complex structure. One might hope that, since the generators are identical, it might be possible to relate the differentials of $\mathcal{C}_*(H_q,J_q)$ with a Morse complex on the total space of $Z \to w^{-1}(0)$. The problem is that flow lines of $JX^q_{H_q}$ may be very different to the gradient flow lines of $-(f|_{w^{-1}(0)} +g)$ that are required to construct such a Morse complex. In particular, it is certainly not true that flow lines of $JX^q_{H_q}$ must be constrained to the submanifold $Z \subset T^*B$. However, as $q$ goes to zero the only flow lines of $JX^q_H$ that contribute to the differential are those that converge to gradient flow lines on $Z$ (to prove this rigorously following \cite{Sch14} we would need to appeal to a recent theorem by Eldering \cite{Eld13} on persistence of noncompact normally hyperbolic invariant manifolds). This implies that in the adiabatic limit $q \to 0$ elements of $W^u_q(p^-) \cap W^s_q(p^+)$ are in bijection with maps $u : \mathbb{R} \to Z$ satisfying
\begin{equation} \label{gradiented}
\frac{du}{dt} = - \nabla (f|_{w^{-1}(0)} + g) (u) , \ \ \ \lim_{t \to \pm \infty} u(t) = p^{\pm}  ,
\end{equation}
where $p^{\pm}$ are the critical points corresponding bijectively to $(x^{\pm},y^{\pm})$. Thus, for $q$ sufficiently small, we obtain an isomorphism of moduli spaces that means we can identify $\mathcal{C}_*(H_q,J_q)$ with a Morse complex on $Z$, whose homology is isomorphic to the singular homology of $w^{-1}(0)$ with grading shifted down by $k$.

\section*{Acknowledgements}
S. Mescher, I. Smith and J. Smith provided many useful comments and insightful suggestions. DST is supported by a Research Fellowship from Peterhouse, Cambridge.

%% The Appendices part is started with the command \appendix;
%% appendix sections are then done as normal sections
%% \appendix

%% \section{}
%% \label{}

%% If you have bibdatabase file and want bibtex to generate the
%% bibitems, please use
%%
%%  \bibliographystyle{elsarticle-num} 
%%  \bibliography{<your bibdatabase>}

%% else use the following coding to input the bibitems directly in the
%% TeX file.

\end{document}